\documentclass[reqno,11pt,a4paper]{article}

\usepackage{amsmath,amssymb,amsthm}
\usepackage[UKenglish]{babel}
\usepackage[top=3.05cm,left=2.54cm,right=2.54cm]{geometry}
\usepackage{multicol}
\usepackage{sectsty}

\sectionfont{\normalsize}
\subsectionfont{\normalsize}

\mathchardef\ordinarycolon\mathcode`\:
\mathcode`\:=\string"8000
\def\vcentcolon{\mathrel{\mathop\ordinarycolon}}
\begingroup \catcode`\:=\active
\lowercase{\endgroup
\let :\vcentcolon}

\theoremstyle{plain}
\newtheorem{theorem}{Theorem}[section]
\newtheorem{lemma}[theorem]{Lemma}
\newtheorem{proposition}[theorem]{Proposition}

\theoremstyle{definition}
\newtheorem{definition}[theorem]{Definition}
\newtheorem{notation}[theorem]{Notation}

\newtheorem{remark}[theorem]{Remark}

\let\origthebibliography=\thebibliography
\def\thebibliography{\renewcommand{\section}[2]{}\origthebibliography}

\newcommand{\algten}{\mathbin{\underline{\otimes}}}
\newcommand{\bop}[1]{\bopp(#1)}
\newcommand{\bopp}{B}
\newcommand{\borel}{\mathcal{B}}
\newcommand{\Cf}{\textit{Cf.}}
\newcommand{\comp}{\mathop{\circ}}
\newcommand{\conv}{\mathop{\star}}
\newcommand{\disc}{\mathcal{D}}
\newcommand{\elltwo}{L^2( \R_+; \mul )}
\newcommand{\eps}{\varepsilon}
\newcommand{\evec}[1]{\evecc(#1)}
\newcommand{\evecc}{\varepsilon}
\newcommand{\evecs}{\mathcal{E}}
\newcommand{\fock}{\mathcal{F}}

\newcommand{\hj}{\alpha}
\newcommand{\hk}{\beta}
\newcommand{\id}{\mathrm{id}}
\newcommand{\ini}{\mathsf{h}}
\newcommand{\mul}{\mathsf{k}}
\newcommand{\qsto}{\mathbin{\Rightarrow}}
\newcommand{\R}{\mathbb{R}}
\newcommand{\rd}{\mathrm{d}}
\newcommand{\std}{\,\rd}
\newcommand{\stlim}{\mathop{\mathrm{st.lim}}\limits}
\newcommand{\uwkten}{\mathbin{\overline{\otimes}}}
\newcommand{\wh}[1]{\widehat{#1}}
\newcommand{\wt}[1]{\widetilde{#1}}

\renewcommand{\ge}{\geqslant}
\renewcommand{\le}{\leqslant}

\begin{document}

\begin{center}
{\LARGE The cocycle identity holds under stopping}
\vspace*{1ex}
\begin{multicols}{2}
{\large Alexander C.~R.~Belton}\\[0.5ex]
{\small Department of Mathematics and Statistics\\
Lancaster University, United Kingdom\\[0.5ex]
\textsf{a.belton@lancaster.ac.uk}}
\columnbreak

{\large Kalyan B.~Sinha}\\[0.5ex]
{\small Jawaharlal Nehru Centre for Advanced\\
Scientific Research, Bangalore, India\\[0.5ex]
\textsf{kbs@jncasr.ac.in}}
\end{multicols}
{\small 20th June 2016}
\end{center}

\begin{abstract}\noindent
In recent work of the authors, it was shown how to use any finite
quantum stop time to stop the CCR flow and its strongly continuous
isometric cocycles (Q.~J.~Math.~65:1145--1164, 2014). The stopped
cocycle was shown to satisfy a stopped form of the cocycle identity,
valid for deterministic increments of the time used for
stopping. Here, a generalisation of this identity is obtained, where
both cocycle parameters are replaced with finite quantum stop times.
\end{abstract}



\section{Introduction}

The history of stopping times in non-commutative probability begins in
1979, with Hudson's work on stopping canonical Wiener processes
\cite{Hud79}. Since then, many authors have contributed to the
subject, and it has developed in various directions and settings:
abstract von~Neumann algebras, to produce first exit times in
$C^*$~algebras and to stop quantum stochastic integrals, for example.
A good introduction for the latter is provided by~\cite{Hud07}; see
\cite{BeS14} for further references.

In this note, we extend a previous result \cite[Theorem~7.2]{BeS14},
which itself built upon work of Parthasarathy and Sinha \cite{PaS87}
and Applebaum~\cite{App88}. Let $V$ be a strongly continuous isometric
cocycle of the CCR flow $\sigma$, so that
\[
V_{s + t} = \wh{V}_s \, \sigma_s( V_t ) \qquad %
\text{for all } s, t \in \R_+,
\]
where $\wh{V}$ is the identity-adapted projection of the $p$-adapted
process $V$. The importance of this identity in classical and quantum
probability is well known; it has an intimate connection with
stochastic integral representation and Feynman--Kac formulae
\cite{Pin72, Bra92}.

If $S$ is a finite quantum stop time then Theorem~7.2 of
\cite{BeS14} gives the stopped cocycle identity
\begin{equation}\label{eqn:cocycle}
V_{S + t} = \wh{V}_S \, \sigma_S( V_t ) \qquad \text{for all } t \in \R_+.
\end{equation}
It is shown below that the following generalisation of
(\ref{eqn:cocycle}) holds: if $T$ is another finite quantum stop time
and the CCR flow $\sigma$ has countable rank then
\[
V_{S \conv T} = \wh{V}_S \, \sigma_S( V_T ),
\]
where $S \conv T$ is the convolution of $S$ and $T$. If $V$ acts on an
initial space $\ini$, it follows that setting
\[
\hj_S( a ) = V_S ( a \otimes I ) V_S^* %
\qquad \text{for all } a \in \bop{\ini}
\]
gives a generalised Evans--Hudson flow $\hj_S$ which satisfies a
non-deterministic version of the mapping-cocycle relation,
\[
\hj_{S \star T} = \wh{\hj}_S \comp \sigma_S \comp \hj_T.
\]

The notation of \cite{BeS14} is followed throughout. In particular,
the algebraic tensor product is denoted by $\algten$, with $\otimes$
the Hilbert-space and $\uwkten$ the ultraweak product. 

\section{Stopped maps with a non-trivial initial space}

In Sections~6 and~7 of \cite{BeS14}, certain maps $E_S$, $\Gamma_S$
and $\sigma_S$ are extended to the case of a non-trivial initial
space, so that the ambient Fock space~$\fock$ is replaced by
$\ini \otimes \fock$, where $\ini$ is a complex Hilbert space. In
order to familiarise the reader with key ideas and notation from
\cite{BeS14}, and as the construction of these extensions are not
quite immediate, the details are provided in this section, together
with some further observations.

\begin{notation}
Let $\fock = \Gamma_+\bigl( \elltwo \bigr)$ be Boson Fock space over
the complex Hilbert space of square-integrable functions on the
half line $\R_+ := [ 0, \infty )$, with values in the complex Hilbert
space $\mul$. Recall the tensor-product decomposition
$\fock = \fock_{t)} \otimes \fock_{[t}$, valid for
all~$t \in ( 0, \infty )$, where
\[
\fock_{t)} := \Gamma_+\bigl( L^2( [ 0, t ); \mul ) \bigr) %
\qquad \text{and} \qquad %
\fock_{t)} := \Gamma_+\bigl( L^2( [ t, \infty ); \mul ) \bigr),
\]
given by extending the identification of exponential vectors such that
$\evec{f} = \evec{f|_{[ 0, t )}} \otimes \evec{f|_{[ t, \infty )}}$
for all $f \in \elltwo$. Let $I$, $I_{t)}$ and $I_{[t}$ denote the
identity operators on $\fock$, $\fock_{t)}$ and $\fock_{[t}$,
respectively, and let $\evecs$ denote the linear span of the set of
exponential vectors in $\fock$.
\end{notation}

\begin{definition}
Let $S$ be a \emph{finite quantum stop time}, so that
$S : \borel[ 0, \infty ] \to \bop{\fock}$ is a map from the Borel
subsets of the extended half line to the set of orthogonal projections
on $\fock$, such that
\begin{itemize}
\item[(i)] the map $A \mapsto \langle x, S( A ) y \rangle$ is a
complex measure for all $x$, $y\in \fock$,
\item[(ii)] the total mass $S\bigl( [ 0, \infty ] \bigr) = I$, with
$S\bigl( \{ \infty \} \bigr) = 0$, and
\item[(iii)] identity adaptedness holds, so that $S( \{ 0 \} ) = 0$
and $S\bigl( [ 0, t ] \bigr) \in \bop{\fock_{t)}} \otimes I_{[t}$ for
all $t \in ( 0, \infty )$.
\end{itemize}
\end{definition}

\begin{notation}
For all $t \in \R_+$, let
$E_t := \Gamma_+( 1_{[ 0, t )} ) \in \bop{\fock}$ be the second
quantisation of the operator obtained by letting this indicator
function act by multiplication, so that $E_t$ is the orthogonal
projection onto $\fock_{t)} \otimes \evec{0|_{[ t, \infty )}}$, and
let $E_\infty := I$.
\end{notation}

\begin{proposition}\label{prp:expn}
Let $\pi = \{ 0 = \pi_0 < \cdots < \pi_{n + 1} = \infty \}$ be a
finite partition of $[ 0, \infty ]$ and let $\ini$ be a complex
Hilbert space. If $\wt{E}_{S, \pi} := I_\ini \otimes E_{S,\pi}$, where
\[
E_{S, \pi} := %
\sum_{j = 1}^{n + 1} S\bigl( ( \pi_{j - 1}, \pi_j ] \bigr) E_{\pi_j},
\]
then $E_{S, \pi} \to E_S$ and
$\wt{E}_{S, \pi} \to \wt{E}_S := I_\ini \otimes E_S$ in the strong
operator topology as $\pi$ is refined, where $E_S$ and $\wt{E}_S$ are
orthogonal projections.
\end{proposition}
\begin{proof}
The proof of \cite[Theorem~3.7]{BeS14} gives that $E_{S, \pi} \to E_S$
strongly on $\evecs$, and thus
$\wt{E}_{S, \pi} \to \wt{E}_S$ strongly on $\ini \algten \evecs$; the
result follows by the density of this last space in
$\ini \otimes \fock$.
\end{proof}

\begin{notation}
For all $s \in \R_+$, let
$\Gamma_s := \Gamma_+( \theta_s ) \in \bop{\fock}$ be the second
quantisation of the isometric right shift, such that
$( \theta_s f )( t ) = 1_{[ s, \infty )}( t ) f( t - s )$ for all
$t \in \R_+$, and let~$\Gamma_\infty := E_0$.
\end{notation}

\begin{proposition}
Let $\pi$ and $\ini$ be as in Proposition~\ref{prp:expn}. If
$\wt{\Gamma}_{S, \pi} := I_\ini \otimes \Gamma_{S,\pi}$,
where
\[
\Gamma_{S, \pi} := \sum_{j = 1}^{n + 1} %
S\bigl( ( \pi_{j - 1}, \pi_j ] \bigr) \Gamma_{\pi_j},
\]
then $\Gamma_{S, \pi} \to \Gamma_S$ and
$\wt{\Gamma}_{S, \pi} \to \wt{\Gamma}_S := I_\ini \otimes \Gamma_S$
in the strong operator topology as $\pi$ is refined, where~$\Gamma_S$
and~$\wt{\Gamma}_S$ are isometries.
\end{proposition}
\begin{proof}
The claims about $\Gamma_{S, \pi}$ and $\Gamma$ follow from the proof
of \cite[Theorem~3.8]{BeS14}, which also gives that
\[
\| ( \wt{\Gamma}_{S, \pi} - \wt{\Gamma}_S ) u \otimes x \| = %
\| u \| \, \| ( \Gamma_{S, \pi} - \Gamma_S ) x \| \to 0
\]
as $\pi$ is refined, for all $u \in \ini$ and $x \in \fock$. As
$\Gamma_{S, \pi}$ and $\Gamma_S$ are isometries, the same is true for
$\wt{\Gamma}_{S, \pi}$ and~$\wt{\Gamma}_S$. Thus
$\wt{\Gamma}_{S, \pi} \to \wt{\Gamma}_S$ strongly on
$\ini \algten \fock$, and so on $\ini \otimes \fock$, since
$\| \wt{\Gamma}_{S, \pi} \| = 1$ for all $\pi$.
\end{proof}

\begin{notation}
For all $t \in \R_+$, let the ultraweakly continuous unital
$*$-homomorphism
\[
\sigma_t : \bop{\fock} \to \bop{\fock}; \ %
X \mapsto I_{t)} \otimes \Gamma_t X \Gamma_t^*,
\]
where $\Gamma_t$ is regarded here as an isometric isomorphism from
$\fock$ to $\fock_{[t}$ with inverse $\Gamma_t^*$, and
let~$\wt{\sigma}_t := \id_{\bop{\ini}} \uwkten \sigma_t$. Recall that
$( \sigma_t : t \in \R_+ )$ is the CCR flow semigroup with rank
$\dim \mul$.
\end{notation}

\begin{notation}
Let $\fock_{S)} := E_S( \fock )$ and $\fock_{[S} := \Gamma_S( \fock )$
be the pre-$S$ and post-$S$ spaces, with identity operators $I_{S)}$
and $I_{[S}$, respectively.
\end{notation}

\begin{theorem}\label{thm:wtsconv}
Let $\ini$ and $\pi$ be as in Proposition~\ref{prp:expn}. If
\[
\wt{\sigma}( Z )_{S, \pi} := %
( \id_{\bop{\ini}{}} \uwkten \sigma_{S,\pi} )( Z ) \qquad %
\text{for all } Z \in \bop{\ini \otimes \fock},
\]
where
\[
\sigma_{S, \pi} : \bop{\fock}{} \to \bop{\fock}{}; \ %
X \mapsto \sum_{j = 1}^{n + 1} %
\sigma_{\pi_j}( X ) \, S\bigl( ( \pi_{j - 1}, \pi_j ] \bigr),
\]
then $\sigma_{S, \pi} \to \sigma_S$ and
$\wt{\sigma}_{S, \pi} \to \wt{\sigma}_S := %
\id_{\bop{\ini}} \uwkten \sigma_S$
pointwise in the strong operator topology as $\pi$ is refined, where
$\sigma_S$ and $\wt{\sigma}_S$ are ultraweakly continuous unital
$*$-homomorphisms. Furthermore, there exist isometric isomorphisms
\[
j_S : \fock_{S)} \otimes \fock_{[S} \to \fock %
\qquad \text{and} \qquad %
\wt{\jmath}_S : \fock_{S)} \otimes \ini \otimes \fock_{[S} \to %
\ini \otimes \fock
\]
such that
\begin{align}
\sigma_S( X ) & = j_S ( I_{S)} \otimes \Gamma_S X \Gamma_S^* ) j_S^*
\qquad \text{for all } X \in \bop{\fock}{} \label{eqn:sj} \\[1ex]
\text{and} \qquad %
\wt{\sigma}_S( Z ) & = \wt{\jmath}_S %
( I_{S)} \otimes \wt{\Gamma}_S Z \wt{\Gamma}_S^* ) \wt{\jmath}_S^* %
\qquad \text{for all } Z \in \bop{\ini \otimes \fock}{}.
\label{eqn:wtsj}
\end{align}
\end{theorem}
\begin{proof}
The convergence of $\sigma_{S, \pi}$ to $\sigma_S$, and the fact that
the latter is a unital $*$-homomorphism, follows from
\cite[Theorem~5.2]{BeS14}. The representation (\ref{eqn:sj}) is
\cite[Proposition~5.3]{BeS14}, and this shows that the map
$X \mapsto \sigma_S( X )$ is continuous when $\bop{\fock}$ is equipped
with the ultraweak topology, as ampliation gives a normal
representation of any von~Neumann algebra. In particular, the map
$\wt{\sigma}_S$ is an ultraweakly continuous unital $*$-homomorphism
such that (\ref{eqn:wtsj}) holds, where the isometric isomorphism
\[
\wt{\jmath}_S : \fock_{S)} \otimes \ini \otimes \fock_{[S} \to %
\ini \otimes \fock; \ %
x \otimes u \otimes y \mapsto u \otimes j_S( x \otimes y ),
\]
because (\ref{eqn:wtsj}) holds if $Z$ is a simple tensor, and both
sides are ultraweakly continuous functions of~$Z$.

It remains to prove that $\wt{\sigma}_{S, \pi}$ converges to
$\wt{\sigma}_S$. Working as in the proof of \cite[Theorem~5.2]{BeS14},
if the finite partition~$\pi'$ is a refinement of~$\pi$, then, for any
$u \in \ini$ and any $f \in \elltwo$ with compact support,
\begin{align*}
\| ( \wt{\sigma}_{S, \pi'} - \wt{\sigma}_{S, \pi} )( Z ) %
u & \evec{f} \| \\[1ex]
 & \le \| S\bigl( [ 0, \pi_n ] \bigr) \evec{f} \| %
\sup\{ \| ( \wt{\sigma}_r( Z ) - Z ) u \evec{f( \cdot + s )} : %
r \in [ 0, \delta_\pi ], \ s \in [ 0, \tau ] \} \\
 & \quad + \| S\bigl( ( \pi_n, \infty ) \bigr) u \evec{f} \| \, %
( \| Z \| + 1 ),
\end{align*}
where $\delta_\pi := \max\{ \pi_j - \pi_{j - 1} : j = 1, \ldots, n \}$
and $f$ has support contained in
$[ 0, \tau ] \subseteq [ 0, \infty )$.

Using the same argument as in the proof of \cite[Theorem~5.2]{BeS14},
and noting that $r \mapsto \wt{\sigma}_r( Z )$ is strongly continuous,
it now follows that $\wt{\sigma}_{S, \pi}( Z ) u \evec{f}$ is
convergent, as $\pi$ is refined, for any~$u \in \ini$ and any
$f \in \elltwo$ with compact support; let the limit be denoted by
$\lambda_S( Z ) u \evec{f}$ and extend by linearity. Since
\[
\| \lambda_S( Z ) z \| = %
\lim_\pi \| \wt{\sigma}_{S, \pi}( Z ) z \| \le %
\| Z \| \, \| z \| \qquad %
\text{for all } z \in \ini \algten \evecs_c,
\]
where $\evecs_c$ is the linear span of those exponential vectors
corresponding to functions with compact support, there exists a
bounded linear operator $\lambda_S( Z )$ on $\ini \otimes \fock$
which extends the linear map~$z \mapsto \lambda_S( Z ) z$. Furthermore,
the usual approximation argument gives that
$\wt{\sigma}_{S, \pi}( Z ) \to \lambda_S( Z )$ in the strong operator
topology, everywhere on $\ini \otimes \fock$.

To conclude, we use (\ref{eqn:wtsj}) and argue as the proof of
\cite[Proposition~5.3]{BeS14}. Using the notation of that proof and
the identity at the top of \cite[p.1158]{BeS14}, that
\[
\langle I_{S, \pi \cap [ 0, t ]}( f, \Gamma \evec{g} ), %
\sigma_{S, \pi}( X ) %
I_{S, \pi \cap [ 0, t ]}( f', \Gamma \evec{g'} ) \rangle = %
\langle E_{S, \pi \cap [ 0, t ]} \evec{f}, %
E_{S, \pi \cap [ 0, t ]} \evec{f'} \rangle \, %
\langle \evec{g}, X \evec{g'} \rangle,
\]
together with the ultraweak continuity of $\wt{\sigma}_{S, \pi}$,
it follows that
\begin{multline*}
\langle u \otimes I_{S, \pi \cap [ 0, t ]}( f, \Gamma \evec{g} ), %
\wt{\sigma}_{S, \pi}( Z ) u' \otimes %
I_{S, \pi \cap [ 0, t ]}( f', \Gamma \evec{g'} ) \rangle \\[1ex]
 = \langle E_{S, \pi \cap [ 0, t ]} \evec{f}, %
E_{S, \pi \cap [ 0, t ]} \evec{f'} \rangle \, %
\langle u \evec{g}, Z u' \evec{g'} \rangle
\end{multline*}
for all $u$, $u' \in \ini$, $f$, $f'$, $g$, $g' \in \elltwo$,
$t \in ( 0, \infty )$ and $Z \in \bop{\ini \otimes \fock}{}$. As $\pi$
is refined, the right-hand side converges to
\[
\langle E_{S, t} \evec{f} \otimes u \otimes \Gamma_S \evec{g}, %
( I_{S)} \otimes \wt{\Gamma}_S Z \wt{\Gamma}_S^* ) %
E_{S, t} \evec{f'} \otimes u' \otimes \Gamma_S \evec{g'} \rangle,
\]
by \cite[Theorem~3.7]{BeS14}, whereas the left-hand side converges to
\begin{multline*}
\langle u \otimes j_S( E_{S, t} \evec{f} \otimes \Gamma_S \evec{g} ),
\lambda_S( Z ) u' \otimes %
j_S( E_{S, t} \evec{f'} \otimes \Gamma_S \evec{g'} ) \rangle \\[1ex]
 = \langle \wt{\jmath}_S( E_{S, t} \evec{f} \otimes u \otimes %
\Gamma_S \evec{g} ), \lambda_S( Z ) \wt{\jmath}_S( %
E_{S, t} \evec{f'} \otimes u' \otimes \Gamma_S \evec{g'} ) \rangle,
\end{multline*}
by \cite[Lemma~3.4 and Theorem~3.10]{BeS14}. The result follows.
\end{proof}

\begin{remark}\label{rem:socts}
The representations (\ref{eqn:sj}) and (\ref{eqn:wtsj}) also give that
$X \mapsto \sigma_S( X )$ and $Z \mapsto \wt{\sigma}_S( Z )$ are
continuous on bounded subsets of $\bop{\fock}{}$ and
$\bop{\ini \otimes \fock}{}$, respectively, when these spaces are
equipped with the strong operator topology, since this is true of the
ampliation map $T \mapsto I \otimes T$.
\end{remark}

\section{The cocycle identity with two stop times}

\begin{definition}[{\cite[Definition~4.1]{BeS14}}]
The \emph{convolution} $S \conv T$ of two finite quantum stop times
$S$ and $T$ is
\[
S \conv T : \borel( \R_+ ) \to \bop{\fock}{}; \ %
A \mapsto ( S \otimes T )\bigl( f^{-1}( A ) \bigr),
\]
where
\[
f : \R_+ \times \R_+ \to \R_+; \ ( x, y ) \mapsto x + y
\]
and
\[
S \otimes T : \borel( \R_+ \times \R_+ ) \to \bop{\fock}{}; \ %
A \times B \mapsto j_S( %
S( A )|_{\fock_{S)}} \otimes \Gamma_S T( B ) \Gamma_S^* ) %
j_S^*.
\]
\end{definition}

\begin{lemma}\label{lem:product}
Let $S$ and $T$ be finite quantum stop times. Then
\begin{equation}\label{eqn:prodmeas}
( S \otimes T )( A \times B ) = %
S( A ) \sigma_S\bigl( T( B ) \bigr) \qquad %
\text{for all } A, B \in \borel[ 0, \infty ].
\end{equation}
\end{lemma}
\begin{proof}
If $t \in [ 0, \infty ]$ then, by Theorems~3.7 and~3.10, together with
Lemma~3.4, of~\cite{BeS14},
\begin{align*}
j_S( S\bigl( [ 0, t ] \bigr) E_S \evec{f} \otimes \Gamma_S x ) = %
j_S( E_{S, t} \evec{f} \otimes \Gamma_S x ) & = %
\int_{[ 0, t ]} S( \rd s ) \evec{f|_{[ 0, s )}} \otimes %
\Gamma_s x \\[1ex]
 & = S\bigl( [ 0, t ] \bigr) \int_{[ 0, \infty ]} S( \rd s ) %
\evec{f|_{[ 0, s )}} \otimes \Gamma_s x \\[1ex]
 & = S\bigl( [ 0, t ] \bigr) %
j_S( E_S \evec{f} \otimes \Gamma_S x )
\end{align*}
for all $f \in \elltwo$ and $x \in \fock$. Hence
\[
j_S( S( A )|_{\fock_{S)}} \otimes I_{[S} ) j_S^* = S( A ) %
\qquad \text{for all } A \in \borel[ 0, \infty ].
\]
It follows that
\begin{align*}
( S \otimes T ) ( A \times B ) & := %
j_S( S( A )|_{\fock_{S)}} \otimes \Gamma_S T( B ) \Gamma_S^* ) %
j_S^* \\[1ex]
& \hphantom{:}= j_S( S( A )|_{\fock_{S)}} \otimes I_{[S} ) j_S^* \, %
j_S( I_{S)} \otimes \Gamma_S T( B ) \Gamma_S^* ) j_S^* = %
S( A ) \sigma_S\bigl( T( B ) \bigr)
\end{align*}
for all $A$, $B \in \borel[ 0, \infty ]$, where the final identity is
a consequence of (\ref{eqn:sj}).
\end{proof}

\begin{remark}\label{rem:prodst}
(i) If the quantum stop times $S$ and $T$ are extended by ampliation
to act on~$\ini \otimes \fock$ then the identity
(\ref{eqn:prodmeas}) becomes
\[
( S \otimes T )( A \times B ) = %
S( A ) \wt{\sigma}_S\bigl( T( B ) \bigr) \qquad %
\text{for all } A, B \in \borel[ 0, \infty ].
\]
This extension will be made when appropriate without further comment.

(ii) If $0 \le p < q < \infty$ and $0 \le r < s < \infty$ then
Theorem~\ref{thm:wtsconv} implies that
\begin{align*}
( S \otimes T )\bigl( ( p, q ] \times ( r, s ] \bigr) & = %
S\bigl( ( p, q ] \bigr) \wt{\sigma}_S( T\bigl( ( r, s ] \bigr) )
\\[1ex]
 & = \stlim_\pi \sum_{j = 1}^m %
S\bigl( ( \pi_{j - 1}, \pi_j ] \bigr) \, %
\wt{\sigma}_{\pi_j}( T \bigl( ( r, s ] \bigr) ) %
\in I_\ini \otimes \bop{\fock_{q + s)}}{} \otimes I_{[q + s},
\end{align*}
where $\pi = \{ p = \pi_0 < \cdots < \pi_m = q \}$ is a typical
finite partition of the interval $[ p, q ]$.
\end{remark}

\begin{lemma}\label{lem:discprod}
Suppose $S$ and $T$ are finite quantum stop times, with $T$ discrete,
so that there exists a finite set
$\{ t_1 < \cdots < t_m \} \subseteq ( 0, \infty )$ such that
$T\bigl( \{ t_1, \ldots, t_m \} \bigr) = I$. Then
\[
( S \conv T )( C ) = \sum_{j = 1}^m S\bigl( ( C - t_j )_+ \bigr) %
\sigma_S( T\bigl( \{ t_j \} \bigr) ) \qquad \text{for all } %
C \in \borel( \R_+ ),
\]
where $( C - t )_+ := \{ s \in \R_+ : s + t \in C \}$ for all
$t \in \R_+$.
\end{lemma}
\begin{proof}
Note first that, by Lemma~\ref{lem:product},
\[
\sum_{j = 1}^m ( S \otimes T )( \R_+ \times \{ t_j \} ) = %
\sum_{j = 1}^m S( \R_+ ) \sigma_S( T\bigl( \{ t_j \} \bigr) ) = %
\sigma_S( T\bigl( \{ t_1, \ldots, t_m \} \bigr) ) = I,
\]
so
\begin{align*}
( S \conv T )( C ) & = ( S \otimes T )%
\bigl( \{ ( x, y ) \in \R_+^2 : x + y \in C \} \bigr) \\[1ex]
 & = \sum_{j = 1}^m ( S \otimes T )%
\bigl( \{ ( x, t_j ) \in \R_+^2 : x + t_j \in C \} \bigr) \\[1ex]
 & = \sum_{j = 1}^m S\bigl( ( C - t_j )_+ \bigr) %
\sigma_S( T\bigl( \{ t_j \} \bigr) ).
\qedhere
\end{align*}
\end{proof}

\begin{remark}
If the discrete stopping time $T$ is supported at one point, so that
$T\bigl( \{ t \} \bigr) = I$ for some $t \in ( 0, \infty )$, then
$S \star T = S + t$, where
$( S + t )( A ) := S\bigl( ( A - t )_+ \bigr)$ for all
$A \in \borel[ 0, \infty ]$.
\end{remark}

\begin{definition}
Let $p \in \bop{\mul}$ be an orthogonal projection and, for all
$t \in \R_+$, let $P_{[t} \in \bop{\fock_{[t}}$ be the orthogonal
projection such that $P_{[t} \evec{f} = \evec{p f}$ for all
$f \in L^2( [ t, \infty ); \mul )$, where~$p$ acts pointwise.

A family of bounded operators
$V = ( V_t )_{t \in \R_+} \subseteq \bop{\ini \otimes \fock}$ is
\emph{$p$-adapted} if
\[
V_t = V_{t)} \otimes P_{[t} \qquad \text{for every } t \in \R_+,
\]
where $V_{t)} \in \bop{\ini \otimes \fock_{t)}}$. If $p = 0$ or
$p = I_\mul$ then $p$-adaptedness is known as vacuum adaptedness or
identity adaptedness, respectively.

Given a $p$-adapted family of bounded operators $V$, the
\emph{identity-adapted projection} $\wh{V}$ is the family of operators
$\wh{V}$, where $\wh{V}_t := V_{t)} \otimes I_{[t}$ for all
$t \in \R_+$.

A $p$-adapted family of bounded operators $V$ is an
\emph{isometric cocycle} if $\wh{V}_t$ is an isometry for
all~$t \in \R_+$ and
\[
V_{s + t} = \wh{V}_s \, \wt{\sigma}_s( V_t ) \qquad %
\text{for all } s, t \in \R_+.
\]
A $p$-adapted isometric cocycle $V$ is \emph{strongly continuous} if
$t \mapsto V_t z$ is continuous for all $z \in \ini \otimes \fock$.
\end{definition}

\begin{theorem}%
{\textup{\cite[Theorem~6.5, Corollary 6.6 and Theorem~7.2]{BeS14}}}
If $S$ is a finite quantum stop time, $V$ is a strongly continuous
isometric $p$-adapted cocycle and
\[
V_{S, \pi} := \sum_{k = 1}^{n + 1} V_{\pi_k} \, %
S\bigl( ( \pi_{k - 1}, \pi_k ] \bigr)
\]
for any finite partition
$\pi = \{ 0 = \pi_0 < \cdots < \pi_{n + 1} = t \}$ of $[ 0, t ]$, then
$V_{S, \pi}$ is a contraction and there exists a contraction
$V_{S, t} \in \bop{\ini \otimes \fock}$ such that
$V_{S, \pi} \to V_{S, t}$ in the strong operator topology as $\pi$ is
refined, for all $t \in ( 0, \infty )$. Furthermore, there exists a
contraction $V_S \in \bop{\ini \otimes \fock}$ such that
$V_{S, t} \to V_S$ in the strong operator topology as $t \to \infty$,
and
\[
V_{S + t} = %
\wh{V}_S \, \wt{\sigma}_S( V_t ) \qquad \text{for all } t \in \R_+.
\]
\end{theorem}
\begin{proof}
The only thing which not immediate is the assertion, at end of
\cite[Proof of Theorem~7.2]{BeS14}, that
\[
\wh{V}_{S, \pi \cap [ 0, t ]} \, \wt{\sigma}_{S, \pi}( Z ) \to %
\wh{V}_{S, t} \, \wt{\sigma}_S( Z )
\]
as the partition $\pi$ is refined, for all
$Z \in \bop{\ini \otimes \fock}$ and $t \ge 0$. (In fact, a very
slightly weaker claim is made.) It follows from
\cite[Theorem~6.5]{BeS14} that
\[
\wh{V}_{S, \pi \cap [ 0, t ]} \to \wh{V}_{S, t}
\]
in the strong operator topology, and
$\| \wh{V}_{S, \pi \cap [ 0, t ]} \| \le 1$ for all $\pi$, by
\cite[Lemma~6.4]{BeS14}, so the claim holds as long as
$\wt{\sigma}_{S, \pi}( Z ) \to \wt{\sigma}_S( Z )$ in the strong
operator topology. However, this is part of Theorem~\ref{thm:wtsconv}.
\end{proof}

\begin{lemma}\label{lem:discrete}
If $S$ and $T$ are finite quantum stop times, with $T$ discrete, then
\[
V_{S \conv T} = \wh{V}_S \, \wt{\sigma}_S( V_T )
\]
for any strongly continuous isometric $p$-adapted cocycle $V$.
\end{lemma}
\begin{proof}
If $T$ is as in the statement of Lemma~\ref{lem:discprod} and
$t > t_m$ then
\begin{align*}
V_{S \conv T, t} & = \stlim_\pi \sum_{k = 1}^{n + 1} V_{\pi_k} \, %
( S \conv T )\bigl( ( \pi_{k - 1}, \pi_k ] \bigr) \\[1ex]
 & = \stlim_\pi \sum_{j = 1}^m \sum_{k = 1}^{n + 1} V_{\pi_k} \, %
S\bigl( ( \pi_{k - 1} - t_j, \pi_k - t_j ]_+ \bigr) %
\wt{\sigma}_S( T\bigl( \{ t_j \} \bigr) ),
\end{align*}
where $\pi = \{ 0 = \pi_0 < \cdots < \pi_{n + 1} = t \}$ and
$( x, y ]_+ = \{ s \in \R_+ : x < s \le y \}$. For $j = 1$, \ldots,
$m$ and $k = 0$, \ldots, $n + 1$, let
\[
\pi^j_k = \left\{ \begin{array}{ll}
 \pi_k - t_j & \mbox{if } \pi_k \ge t_j, \\[1ex]
 0 & \mbox{otherwise},
\end{array}\right.
\]
so that $\pi^j$ is a partition of $[ 0, t - t_j ]$. Then
\begin{align*}
\sum_{j = 1}^m \sum_{k = 1}^{n + 1} V_{\pi_k} \, & %
S\bigl( ( \pi_{k - 1} - t_j, \pi_k - t_j ]_+ \bigr) %
\wt{\sigma}_S( T\bigl( \{ t_j \} \bigr) ) \\[1ex]
 & = \sum_{j = 1}^m \sum_{k = 1}^{n + 1} V_{\pi^j_k + t_j} %
S\bigl( ( \pi^j_{k - 1}, \pi^j_k ] \bigr) %
\wt{\sigma}_S( T\bigl( \{ t_j \} \bigr) ) \\[1ex]
 & = \sum_{j = 1}^m \sum_{k = 1}^{n + 1} \wh{V}_{\pi^j_k} \, %
S\bigl( ( \pi^j_{k - 1}, \pi^j_k ] \bigr) %
\sum_{l = 1}^{n + 1} S\bigl( ( \pi^j_{l - 1}, \pi^j_l ] \bigr) %
\wt{\sigma}_{\pi^j_l}( V_{t_j} ) %
\wt{\sigma}_S( T\bigl( \{ t_j \} \bigr) ) \\[1ex]
 & \to \sum_{j = 1}^m \wh{V}_{S, t - t_j} \wt{\sigma}_S( V_{t_j} ) %
\wt{\sigma}_S\bigl( T( \{ t_j \} ) \bigr)
\end{align*}
in the strong operator topology as $\pi$ is refined; for the final
identity, note that
\[
V_{s + t} S\bigl( ( r, s ] \bigr) = %
\wh{V}_s \, \wt{\sigma}_s( V_t ) S\bigl( ( r, s ] \bigr) = %
\wh{V}_s \, S\bigl( ( r, s ] \bigr) \wt{\sigma}_s( V_t ) %
\qquad \text{whenever} \quad 0 \le r < s < t < \infty.
\]
Hence
\[
V_{S \conv T} = \stlim_{t \to \infty} V_{S \conv T, t} = %
\stlim_{t \to \infty} \sum_{j = 1}^m \wh{V}_{S, t - t_j} \, %
\wt{\sigma}_S( V_{t_j} T\bigl( \{ t_j \} \bigr) ) = %
\wh{V}_S \, \wt{\sigma}_S( V_T ).
\qedhere
\]
\end{proof}

\begin{definition}[{\Cf~\cite[p.322]{PaS87}}]\label{dfn:qstconv}
A sequence of finite quantum stop times $( S_n )_{n \ge 1}$ is said to
\emph{converge} to a quantum stop time $S$, written $S_n \qsto S$, if
$S_n\bigl( [ 0, t ] \bigr) \to S\bigl( [ 0, t ] \bigr)$ in the strong
operator topology for all but a countable set of points $t \in \R_+$.
\end{definition}

\begin{lemma}\label{lem:cocyclecty}
Let $V$ be a strongly continuous isometric $p$-adapted cocycle. If
$( S_n )_{n \ge 1}$ is a sequence of finite quantum stop times such
that $S_n \qsto S$ for some finite quantum stop time $S$ then
$V_{S_n} \to V_S$ in the strong operator topology.
\end{lemma}
\begin{proof}
The usual approximation argument shows it suffices to prove that
 $\| ( V_{S_n} - V_S ) u\evec{f} \| \to 0$ as $n \to \infty$, where
$u \in \ini$ and $f \in \elltwo$ are arbitrary.

From the proof of \cite[Corollary~6.6]{BeS14}, if $S$ is any finite
quantum stop time and $s$, $t \in \R_+$ are such that $s \le t$ then
\[
\| ( V_{S, t} - V_{S, s} ) u \evec{f} \| \le %
\| S\bigl( ( s, t ] \bigr) u \evec{f} \|.
\]
Letting $t \to \infty$ and recalling that
$S\bigl( \{ \infty \} \bigr) = 0$, it follows that
\[
\| ( V_S - V_{S, s} ) u \evec{f} \| \le %
\| S\bigl( ( s, \infty ) \bigr) u \evec{f} \|.
\]
Furthermore, from the proof of \cite[Theorem~6.5]{BeS14},
\[
\| ( V_{S, \pi'} - V_{S, \pi} ) u \evec{f} \| \le %
\sup\{ \| ( V_r - V_{\pi_j} ) u \evec{f} \| : %
r \in [ \pi_j, \pi_{j + 1} ], \ j = 0, \ldots, m \} \, %
\| S\bigl( [ 0, s ] \bigr) \evec{f} \|,
\]
where $\pi'$ is any refinement of the partition
$\pi = \{ 0 = \pi_0 < \cdots < \pi_{m + 1} = s \}$; refining $\pi'$
shows that the same inequality holds with $V_{S, \pi'}$ replaced
by~$V_{S, s}$. Hence
\begin{multline*}
\| ( V_S - V_{S, \pi} ) u \evec{f} \|  \\[1ex]
\le \sup\{ \| ( V_r - V_{\pi_j} ) u \evec{f} \| : %
r \in [ \pi_j, \pi_{j + 1} ], \ j = 0, \ldots, m \} \, %
\| \evec{f} \| + \| S\bigl( ( s, \infty ) \bigr) u \evec{f} \|.
\end{multline*}
Now fix $\eps > 0$, choose $s \in \R_+$ such that
$S_n\bigl( [ 0, s ] \bigr) \to S\bigl( [ 0, s ] \bigr)$ in the strong
operator topology and
$\| S\bigl( ( s, \infty ) \bigr) u \evec{f} \| < \eps$, and
note that $\| S_n\bigl( ( s, \infty ) \bigr) u \evec{f} \| < \eps$
for all sufficiently large $n$. Therefore
\begin{align*}
\| ( V_S - V_{S_n} ) u \evec{f} \| & \le %
\| ( V_S - V_{S, \pi} ) u \evec{f} \| + %
\| ( V_{S_n} - V_{S_n, \pi} ) u \evec{f} \| + %
\| ( V_{S, \pi} - V_{S_n, \pi} ) u \evec{f} \| \\[1ex]
 & < %
2 \sup\{ \| ( V_r - V_{\pi_j} ) u \evec{f} \| : %
r \in [ \pi_j, \pi_{j + 1} ], \ j = 0, \ldots, m \} \, %
\| \evec{f} \| \\
 & \qquad + 2 \eps + %
\| ( V_{S, \pi} - V_{S_n, \pi} ) u \evec{f} \| \\[1ex]
 & < 4 \eps + \| ( V_{S, \pi} - V_{S_n, \pi} ) u \evec{f} \|
\end{align*}
as long as $\pi$ is chosen to be sufficiently fine, so that
\[
\sup\{ \| ( V_r - V_{\pi_j} ) u \evec{f} \| : %
r \in [ \pi_j, \pi_{j + 1} ], \ j = 0, \ldots, m \} \, %
\| \evec{f} \|  < \eps.
\]
Finally, if $\pi$ is chosen so
$S_n\bigl( [ 0, \pi_j ] \bigr) \to S\bigl( [ 0, \pi_j ] \bigr)$
in the strong operator topology as~$n \to \infty$, for~$j = 0$,
\ldots, $m + 1$, then, since
\[
\| ( V_{S, \pi} - V_{S_n, \pi} ) u \evec{f} \| \le %
\sum_{j = 1}^{m + 1} \| V_{\pi_j} \, %
\bigl( S( \pi_{j - 1}, \pi_j] ) - S_n( \pi_{j - 1}, \pi_j ] \bigr) %
u \evec{f} \| \to 0
\]
as $n \to \infty$, and $\eps$ is arbitrary, the result follows.
\end{proof}

\begin{lemma}\label{lem:discapprox}
Let $T$ be a finite quantum stop time, and suppose that the
multiplicity space $\mul$ is separable. There exists a sequence of
discrete quantum stop times $( T_n )_{n \ge 1}$ such that
$T_n \qsto T$. Furthermore, $S \conv T_n \qsto S \conv T$ for any
finite quantum stop time $S$.
\end{lemma}
\begin{proof}
As is well known, a spectral measure is strongly right continuous with
left limits: if~$x \in \fock$ then
\[
\lim_{s \to t+} S\bigl( [ 0, t ] \bigr) x - %
S\bigl( [ 0, s ] \bigr) x = %
-\lim_{s \to t+} S\bigl( ( t, s ] \bigr) x = 0,
\]
whereas
\[
\lim_{s \to t-} S\bigl( [ 0, t ] \bigr) x - %
S\bigl( [ 0, s ] \bigr) x = %
\lim_{s \to t-} S\bigl( ( s, t ] \bigr) x = S\bigl( \{ t \} \bigr) x.
\]
In particular, the set of discontinuities
$\disc_S( x ) := \{ t \in \R_+ : S\bigl( \{ t \} \bigr) x \neq 0 \}$
is countable.

Now suppose $\{ x_n : n \ge 1 \}$ is dense in $\fock$ and let
$\disc_S := \cup_{n \ge 1} \disc_S( x_n )$. An $\eps / 3$ argument
shows that~$t \mapsto S\bigl( [ 0, t ] \bigr) x$ is continuous on
$\R_+ \setminus \disc$ for all $x \in \fock$, so $\disc_S$ is the set
of discontinuities of~$S$ on~$\R_+$.

For all $n \ge 1$, let the finite partition
$\pi^n = \{ 0 = \pi^n_0 < \pi^n_1 < \cdots < \pi^n_n < \infty \}$
be such that~$\pi^n_n \to \infty$ and 
$\max\{ \pi^n_k - \pi^n_{k - 1} : k = 1, \ldots, n \} \to 0$
as $n \to \infty$. Define a discrete quantum stop time
\[
T_n : \borel[ 0, \infty ] \to \bop{\fock}; \ %
A \mapsto \sum_{k = 1}^{n - 1} %
1_A( \pi^n_k ) T\bigl( ( \pi^n_{k - 1}, \pi^n_k ] \bigr) + %
1_A( \pi^n_n ) T\bigl( ( \pi^n_n, \infty ] \bigr),
\]
and note that
\[
T_n\bigl( [ 0, t ] \bigr) = \left\{ \begin{array}{ll}
 0 & \text{if } 0 \le t < \pi^n_1, \\[1ex]
 T\bigl( [ 0, \pi^n_1 ] \bigr) & %
\text{if } \pi^n_1 \le t < \pi^n_2, \\[1ex]
 \vdots & \vdots \\[1ex]
 T\bigl( [ 0, \pi^n_{n - 1} ] \bigr) & %
\text{if } \pi^n_{n - 1} \le t < \pi^n_n, \\[1ex]
 I & \text{if } \pi^n_n \le t \le \infty.
\end{array}\right.
\]
Thus if $x \in \fock$ and $t \in \R_+$ then
$\pi^n_k > t \ge \pi^n_{k - 1}$ for some $k \in \{ 1, \ldots, n \}$
once $n$ is sufficiently large, so
\[
T\bigl( [ 0, t ] \bigr) x - T_n\bigl( [ 0, t ] \bigr) x = %
T\bigl( ( \pi^n_{k - 1}, t ] \bigr) x \to T\bigl( \{ t \} \bigr) x
\]
as $n \to \infty$. This last term equals $0$ if
$t \in \R_+ \setminus \disc_T$, and thus $T_n \qsto T$.

For the final claim, let $S$ be a finite quantum stop time. If
$t \in \R_+$ and $f \in \elltwo$ then, by \cite[Corollary~3.5]{BeS14},
\begin{align*}
I_n( t ) & := %
\| \bigl( ( S \conv T_n )\bigl( [ 0, t ] \bigr) - %
( S \conv T )\bigl( [ 0, t ] \bigr) \bigr) \evec{f} \|^2 \\[1ex]
 & \hphantom{:}= \int_{[ 0, t ]} \| %
( T - T_n )\bigl( [ 0, t - s ] \bigr) %
\Gamma_s^* \evec{f} \|^2 %
\exp\Bigl( -\int_s^\infty \| f( u ) \|^2 \std u \Bigr) %
\| S( \rd s ) \evec{f} \|^2.
\end{align*}
To prove that $S \conv T_n \qsto S \conv T$, it suffices to show that
$I_n( t ) \to 0$ as $n \to \infty$ for all but countably
many~$t \in \R_+$, by the usual approximation argument.

Now, as $n \to \infty$, so
$( T - T_n )\bigl( [ 0, t - s ] \bigr) \to %
T\bigl( \{ t - s \} \bigr)$,
by the previous working. Thus the dominated convergence theorem gives
that
\[
I_n( t ) \to I( t ) := \sum_{r \in \disc \cap [ 0, t ] } %
\| T\bigl( \{ r \} \bigr) \Gamma_{t - r}^* \evec{f} \|^2 %
\exp\Bigl( -\int_{t - r}^\infty \| f( u ) \|^2 \std u \Bigr) %
\| S\bigl( \{ t - r \} \bigr) \evec{f} \|^2. 
\]
Thus $I( t ) = 0$ whenever
$t \not \in \disc_S + \disc_T := %
\{ s + r : s \in \disc_S, \ r \in \disc_T \}$
and the result follows.
\end{proof}

\begin{remark}
If the multiplicity space $\mul$ is not separable, the statement of
\cite[Corollary~3.5]{BeS14} requires strong measurability, not just
Borel measurability, of $F$ and $G$. As $t \mapsto \Gamma_t$ and
$t \mapsto \Gamma_t^*$ are strongly continuous and
$t \mapsto S\bigl( [ 0, t ] \bigr)$ is strongly right continuous on
$\R_+$, all the subsequent proofs in \cite{BeS14} remain valid.
\end{remark}

\begin{remark}
It is straightforward to construct on a non-separable Hilbert space a
spectral measure which has an uncountable set of discontinuities. Thus
the separability hypothesis in Lemma~\ref{lem:discapprox} may not be
dropped.
\end{remark}

\begin{theorem}\label{thm:stoppedcocycle}
Let $V$ be a strongly continuous isometric $p$-adapted cocycle, and
suppose that the multiplicity space $\mul$ is separable. If $S$
and $T$ are finite quantum stop times then
\[
V_{S \conv T} = \wh{V}_S \, \wt{\sigma}_S( V_T ).
\]
\end{theorem}
\begin{proof}
By Lemma~\ref{lem:discapprox}, there exists a sequence of discrete
quantum stop times $( T_n )_{n \ge 1}$ such that $T_n \qsto T$ and
$S \conv T_n \qsto S \conv T$. Hence
$V_{S \conv T_n} \to V_{S \conv T}$ in the strong operator topology,
by Lemma~\ref{lem:cocyclecty}. Furthermore,
$V_{S \conv T_n} = V_S \, \wt{\sigma}_S( V_{T_n} )$ for all $n \ge 1$,
by Lemma~\ref{lem:discrete}, so the result follows from another
application of Lemma~\ref{lem:cocyclecty} together with
Remark~\ref{rem:socts}, that $\sigma_S$ is strong operator continuous
on bounded sets.
\end{proof}

The next two theorems show that stopping an isometric cocycle can be
used to produce a form of inner non-unital Evans--Hudson flow.

\begin{theorem}\label{thm:ehid}
Let $V$ be a strongly continuous isometric identity-adapted cocycle.
The map
\[
\hj_S : \bop{\ini} \to \bop{\ini \otimes \fock}; \ %
a \mapsto V_S ( a \otimes I ) V_S^*
\]
is a $*$-homomorphism for any finite quantum stop time $S$.
Furthermore, if the multiplicity space~$\mul$ is separable, the
identity
\begin{equation}\label{eqn:cocycleid}
\hj_{S \star T} = \wh{\hj}_S \comp \wt{\sigma}_S \comp \hj_T
\end{equation}
holds for any finite quantum stop times $S$ and $T$, where
\[
\wh{\hj}_S : %
\bop{\ini \otimes \fock} \to \bop{\ini \otimes \fock}; \ %
X \mapsto V_S X V_S^*.
\]
\end{theorem}
\begin{proof}
Note that $V_S^* V_S = I_\ini \otimes I$, by
\cite[Proposition~6.8]{BeS14}. Thus if $a$, $b \in \bop{\ini}$ then
\[
\hj_S( a ) \hj_S( b ) = %
V_S ( a \otimes I ) V_S^* V_S ( b \otimes I ) V_S^* = %
V_S ( a b \otimes I ) V_S = \hj_S( a b ),
\]
so $\hj_S$ is multiplicative. Linearity and $*$-preservation are
immediate.

For the second claim, note that $\wh{V} = V$. Hence, by
Theorem~\ref{thm:stoppedcocycle}, if $a \in \bop{\ini}$ then
\begin{align*}
\hj_{S \star T}( a ) = %
V_{S \star T} ( a \otimes I ) V_{S \star T}^* & = %
V_S \, \wt{\sigma}_S( V_T ) ( a \otimes I ) %
\wt{\sigma}_S( V_T^* ) \, V_S^* \\[1ex]
 & = V_S \, \wt{\sigma}_S( V_T ( a \otimes I ) V_T^* ) %
\, V_S^* \\[1ex]
 & = ( \wh{\hj}_S \comp \wt{\sigma}_S \comp \hj_T )( a );
\end{align*}
the penultimate equality holds because $\sigma_S$ is unital, so
$\wt{\sigma}_S( a \otimes I ) = a \otimes I$.
\end{proof}

\begin{theorem}\label{thm:ehvac}
Let $V$ be a strongly continuous isometric vacuum-adapted cocycle.
The map
\[
\hk_S : \bop{\ini} \to \bop{\ini \otimes \fock}; \ %
a \mapsto V_S ( a \otimes E_S ) V_S^*
\]
is a $*$-homomorphism for any finite quantum stop time $S$.
Furthermore, if the multiplicity space~$\mul$ is separable, the
identity
\begin{equation}\label{eqn:cocyclevac}
\hk_{S \star T} = \wh{\hk}_S \comp \wt{\sigma}_S \comp \hk_T
\end{equation}
holds for any finite quantum stop times $S$ and $T$, where
\[
\wh{\hk}_S : %
\bop{\ini \otimes \fock} \to \bop{\ini \otimes \fock}; \ %
X \mapsto \wh{V}_S X \wh{V}_S^*.
\]
\end{theorem}
\begin{proof}
Note that $V_S^* V_S = \wt{E}_S$, by \cite[Proposition~6.7]{BeS14}.
Thus if $a$, $b \in \bop{\ini}$ then
\[
\hk_S( a ) \hk_S( b ) = %
V_S ( a \otimes E_S ) V_S^* V_S ( b \otimes E_S ) V_S^* = %
V_S ( a b \otimes E_S ) V_S = \hk_S( a b ),
\]
so $\hk_S$ is multiplicative. As above, linearity and $*$-preservation
are immediate.

For the second claim, let $a \in \bop{\ini}$ and note that, by
Theorem~\ref{thm:stoppedcocycle},
\begin{align*}
\hk_{S \star T}( a ) = %
V_{S \star T} ( a \otimes E_{S \star T} ) V_{S \star T}^* & = %
\wh{V}_S \, \wt{\sigma}_S( V_T ) ( a \otimes E_{S \star T} ) %
\wt{\sigma}_S( V_T^* ) \, \wh{V}_S^* \\[1ex]
 & = \wh{V}_S \, \wt{\sigma}_S( V_T ( a \otimes E_T ) V_T^* ) %
\, \wh{V}_S^* \\[1ex]
 & = ( \wh{\hk}_S \comp \wt{\sigma}_S \comp \hk_T )( a );
\end{align*}
for the penultimate equality, note that
$E_{S \star T} = \sigma_S( E_T )$, by \cite[Theorem~5.4]{BeS14}, which
implies immediately that
$\wt{\sigma}_S( a \otimes E_T ) = a \otimes E_{S \star T}$.
\end{proof}

\begin{remark}
In the context of Theorems~\ref{thm:ehid} and~\ref{thm:ehvac}, note
that $\hj_S( I_\ini ) = V_S V_S^* = \hk_S( I_\ini )$. The former
identity is immediate, and the latter holds because
$V_{S, \pi} \wt{E}_{S, \pi'} V_{S, \pi} = V_{S, \pi} V_{S, \pi}^*$ if
$V$ is vacuum adapted, where $\pi$ is any finite partition of
$[ 0, t ]$ and $\pi'$ is its one-point extension to a partition
of~$[ 0, \infty ]$.
\end{remark}

\begin{remark}
If the finite quantum stop time $S$ is deterministic, so that
$S\bigl( \{ s \} \bigr) = I$ for some $s \in ( 0 , \infty )$, then
$V_S = V_s$ and $\wt{\sigma}_S = \wt{\sigma}_s$. It follows that
(\ref{eqn:cocycleid}) and (\ref{eqn:cocyclevac}) are the stop-time
generalisation of the deterministic mapping-cocycle relation
\cite{Bra92}
\[
\hj_{s + t} = \wh{\hj}_s \comp \wt{\sigma}_s \comp \hj_t %
\qquad \text{for all } s, t \ge 0.
\]
\end{remark}

\subsection*{Acknowledgements}
This work was begun during a conference held at the Kerala School of
Mathematics, Kozhikode, India; the warm hospitality and stimulating
atmosphere provided by the organisers is gratefully acknowledged. It
continued during visits of the first author to the Indian Statistical
Institute, Kolkata, and the Jawaharlal Nehru Centre for Advanced
Scientific Research, Bangalore, with travel supported by the ANCM
project of the Indian Statistical Institute. The second author
gratefully acknowledges support by a grant from the SERB-Distinguished
Fellowship of the Department of Science and Technology, Government of
India.

\end{document}